\documentclass[a4paper,12pt]{article}

\usepackage{tikz}
\usepackage{amssymb}
\usepackage{mathrsfs}
\usepackage{booktabs,caption,fixltx2e}
\usepackage[flushleft]{threeparttable}
\usepackage{mathtools}
\usepackage{amsfonts}%%
\usepackage{amsmath}
\usepackage{amsthm}
 \usepackage{amssymb}
  \usepackage{mathtools}
\usepackage[all,cmtip]{xy}

\usepackage[T1]{fontenc}
\usepackage[utf8]{inputenc}
\usepackage{authblk}  %%%%% author
\newcommand{\CM}{{\mathrm{CM}}}
\newcommand{\Aut}{{\mathrm{Aut}\,}}

\newcommand{\Mon}{{\mathrm{Mon}\,}}

\newtheorem{theorem}{Theorem}%[section]
\newtheorem{lemma}[theorem]{Lemma}
\newtheorem{proposition}[theorem]{Proposition}
\newtheorem{corollary}[theorem]{Corollary}

\theoremstyle{definition}
\newtheorem{definition}{Definition}%[section]
\newtheorem{example}{Example}%[section]
\newtheorem{problem}[theorem]{Problem}

\newtheorem{remark}{Remark}%[section]
\begin{document}
\title{Complete regular dessins and skew-morphisms of cyclic groups}
\author[1]{Yan-Quan Feng\thanks{yqfeng@bjtu.edu.cn}}
\author[2,3]{Kan Hu\thanks{hukan@zjou.edu.cn}}
\author[4,5]{Roman Nedela\thanks{nedela@savbb.sk}}
\author[6]{Martin \v{S}koviera\thanks{skoviera@dcs.fmph.uniba.sk}}
\author[2,3]{Na-Er Wang\thanks{wangnaer@zjou.edu.cn}}
\affil[1]{Department of Mathematics, Beijing Jiaotong University, Beijing 100044, People's Republic of China}
\affil[2]{School of Mathematics, Physics and Information Science, Zhejiang Ocean University, Zhoushan, Zhejiang 316022, People's Republic of China}
\affil[3]{Key Laboratory of Oceanographic Big Data Mining \& Application of Zhejiang Province, Zhoushan, Zhejiang 316022, People's Republic of China}
\affil[4]{University of West Bohemia, NTIS FAV, Pilsen, Czech Republic}
\affil[5]{Mathematical Institute, Slovak Academy of Sciences, Bansk\'a Bystrica, Slovak Republic}
\affil[6]{Department of Informatics, Comenius University, 842 48 Bratislava, Slovak Republic}
\date{}
\maketitle

\begin{abstract}
A dessin is a 2-cell embedding of a connected $2$-coloured bipartite graph
into an orientable closed surface. A dessin is regular if its group of
orientation- and colour-preserving automorphisms acts regularly
on the edges. In this paper we
study regular dessins whose underlying graph is a complete
bipartite graph $K_{m,n}$, called $(m,n)$-complete regular dessins. The
purpose is to establish a rather surprising correspondence
between $(m,n)$-complete regular dessins and pairs of skew-morphisms of
cyclic groups. A skew-morphism of a finite group $A$ is a
bijection $\varphi\colon A\to A$ that satisfies the identity
$\varphi(xy)=\varphi(x)\varphi^{\pi(x)}(y)$ for some function
$\pi\colon A\to\mathbb{Z}$ and fixes the neutral element
of~$A$. We show that every $(m,n)$-complete regular dessin
$\mathcal{D}$ determines a pair of reciprocal
skew-morphisms of the cyclic groups $\mathbb{Z}_n$ and $\mathbb{Z}_m$.
 Conversely, $\mathcal{D}$ can be reconstructed from such a reciprocal pair.
  As a consequence, we prove that complete regular
dessins, exact bicyclic groups with a distinguished pair of
generators, and pairs of  reciprocal skew-morphisms of cyclic
groups are all in one-to-one correspondence. Finally, we
apply the main result to determining all pairs of integers $m$
and $n$ for which there exists, up to interchange of colours,
exactly one $(m,n)$-complete regular dessin. We show that the
latter occurs precisely when every group expressible as a
product of cyclic groups of order $m$ and $n$ is abelian, which
eventually comes down to the condition
$\gcd(m,\phi(n))=\gcd(\phi(m),n)=1$, where $\phi$ is Euler's
totient function.\\[2mm]
\end{abstract}

\section{Introduction}
A \textit{dessin} is a cellular embedding $i:\Gamma\hookrightarrow C$ of a connected bipartite graph $\Gamma$, endowed with
a fixed proper $2$-colouring of its vertices, into an orientable closed surface $C$ such that each
component of $C\setminus i(\Gamma)$ is homeomorphic to the open disc.
An automorphism of a dessin is a colour-preserving automorphism of the underlying
graph that extends to an orientation-preserving self-homeomorphism of
the supporting surface. The action of the automorphism group of
a dessin on the edges is well known to be semi-regular; if this
action is transitive, and hence regular, the dessin itself is
called regular.

Dessins -- more precisely \textit{dessins d'enfants} --
were introduced by Grothendieck in \cite{Gro1997} as a
combinatorial counterpart of algebraic curves. Grothendieck was
inspired by a theorem of Bely\v{\i}~\cite{Be79} which states
that a compact Riemann surface~$C$, regarded as a projective
algebraic curve, can be defined by an algebraic equation $P(x,y)=0$ with
coefficients from the algebraic number field $\mathbb{\bar Q}$
if and only if there exists a non-constant meromorphic function
$\beta:C\to\mathbb{P}^1(\mathbb{C})$, branched over at most
three points, which can be chosen to be $0$, $1$, and $\infty$.
It follows that each such curve carries a dessin in which the
black and the white vertices are the preimages of $0$ and $1$,
respectively, and the edges are the preimages of the unit
interval $I=[0,1]$. The absolute Galois group
$\mathbb{G}=\mathrm{Gal}(\mathbb{\bar Q}/\mathbb{Q})$ has a
natural action on the curves and thus also on the dessins. As
was shown by Grothendieck~\cite{Gro1997}, the action of
$\mathbb{G}$ on dessins is faithful. More recently,
Gonz\'alez-Diez and Jaikin-Zapirain~\cite{GJ2015} have proved
that this action remains faithful even when restricted to
regular dessins. It follows that one can study the absolute
Galois group through its action on such simple and symmetrical
combinatorial objects as regular dessins.

In this paper we study regular dessins whose underlying graph
is a complete bipartite graph $K_{m,n}$, which we call
\textit{complete regular dessins}, or more specifically
\textit{$(m,n)$-complete regular dessins}. The associated algebraic curves
may be viewed as a generalisation of the Fermat
curves, defined by the equation $x^n+y^n=1$ (see Lang
\cite{La1972}). These curves have recently attracted
considerable attention, see for example
\cite{CJSW2009,Jones2010,JS1996,JSW2007,JW2016}.
Classification of complete regular dessins is therefore a very
natural problem, interesting from algebraic, combinatorial, and
geometric points of view.

Jones, Nedela, and \v{S}koviera \cite{JNS2008} were first to
observe that there is a correspondence between complete regular
dessins and exact bicyclic groups. Recall that a finite group
$G$ is \textit{bicyclic} if it can be expressed as a product
$G=AB$ of two cyclic subgroups $A$ and~$B$; if $A\cap B=1$, the
bicyclic group is called \textit{exact}. Exact bicyclic groups
are, in turn, closely related to skew-morphisms of the cyclic groups.

A skew-morphism of a finite group $A$
is a bijection $\varphi\colon A\to A$ fixing the identity
element of $A$ and obeying the morphism-type rule
$\varphi(xy)=\varphi(x)\varphi^{\pi(x)}(y)$ for some integer function
$\pi\colon A\to\mathbb{Z}$.  In the case where $\pi$ is the constant function
$\pi(x)=1$, a skew-morphism is just an automorphism. Thus, skew-morphisms
may be viewed as  a generalisation of group automorphisms. The concept of skew-morphism was introduced by Jajcay and \v{S}ir\'a\v{n} as an algebraic tool to the investigation of regular Cayley maps~\cite{JS2002}. In the seminal paper they proved that a Cayley map $\CM(A,X,P)$ is regular if and only if there is a skew-morphism of $A$ such that $\varphi\restriction_X=P$. Thus
the classification of regular Cayley maps of a finite group $A$ is essentially a problem of determining certain skew-morphisms of $A$. Since then, the theory of skew-morphism has become a dispensable and powerful tool for the study of regular Cayley maps; the interested reader is referred to~\cite{CT2014,CJT2016, JN2015,KK2016,KK2017,KK2005,KMM2013, KKF2006, KWON2013, Zhang2015,Zhang20152} for the up-to-date progress in this direction.

The main purpose of this paper is to establish another unexpected surprising
connection between skew-morphisms and complete regular dessins.
As we already mentioned above every $(m,n)$-complete regular dessin can be represented
as an exact bicyclic group $G=\langle a\rangle\langle b\rangle$ with two distinguished generators
$a$ and $b$ of orders $m$ and $n$, respectively. The factorisation gives
rise to a pair of closely related skew-morphisms of cyclic
groups $\varphi\colon \mathbb{Z}_n\to\mathbb{Z}_n$ and $\varphi^*\colon \mathbb{Z}_m\to \mathbb{Z}_m$ which satisfy two simple technical conditions (see
Definition~\ref{def:recipr}); such a pair of skew-morphisms
will be called \textit{reciprocal}. We prove that isomorphic
complete regular dessins give rise to the \textit{same} pair of
reciprocal skew-morphisms, which is a rather remarkable fact,
because every complete regular dessin thus receives a natural
algebraic invariant.

Even more surprising is the fact that
given a  pair of reciprocal skew-morphisms
$\varphi\colon\mathbb{Z}_n\to \mathbb{Z}_n$ and
$\varphi^*\colon\mathbb{Z}_m\to \mathbb{Z}_m$, one can
reconstruct the original complete regular dessin up to
isomorphism. In other words, a  pair of reciprocal
skew-morphisms of the cyclic groups constitutes a complete set of
invariants for a regular dessin whose underlying graph is a
complete bipartite graph. One can therefore study and classify
complete regular dessins by means of determining pairs
of reciprocal skew-morphisms of cyclic groups. This connection
provides the strongest motivation for the prominent open problem of
determining skew-morphisms of the cyclic groups studied in \cite{BJ2014,BJ2016,CJT2016, KN2011,ZD2016}.

The relationship between
complete regular dessins and exact bicyclic groups mentioned
above also has important implications for the
classical classification problem of bicyclic groups in group theory
(see \cite{Huppert1953-1, Huppert1953,Janko2008} for instance). More
precisely, suppose that we are given an exact product $G=AB$ of
two cyclic groups $A$ and $B$ with distinguished generators
$a\in A$ and $b\in B$. The corresponding  pair of reciprocal
skew-morphisms $(\varphi,\varphi^*)$ and associated power functions $(\pi,\pi^*)$
can be alternatively derived from the equations
\[
ba^x=a^{\varphi(x)}b^{\pi(x)}\quad\text{and}\quad ab^y=b^{\varphi^*(y)}a^{\pi^*(y)},
\]
 and thus
encodes the commuting rules within $G$. By our main result,
determining all exact bicyclic groups with a distinguished
generator pair is equivalent to determining all
pairs of reciprocal skew-morphisms. Thus to describe all exact
bicyclic groups it is sufficient to characterise all pairs of
reciprocal skew-morphisms of the cyclic groups.

Our paper is organised as follows. In Section~2 we describe the
basic correspondence between the complete regular dessins and
bicyclic triples $(G;a,b)$, where $G$ is a group which
factorises as $G=\langle a\rangle\langle b\rangle$ with
$\langle a\rangle\cap\langle b\rangle=1$. Given a complete
regular dessin~$\mathcal D$, its automorphism group
$G=\Aut({\mathcal D})$ can be factorised as a product of two disjoint
 cyclic subgroups $\langle a\rangle$ and $\langle b\rangle$ where
 $\langle a\rangle$ is the stabiliser of one black vertex
and $\langle b\rangle$ is the stabiliser of one white vertex.
The triple $(G;a,b)$ is then an exact bicyclic triple.
Conversely, each exact bicyclic triple $(G;a,b)$
determines a complete regular dessin where the elements of $G$
are the edges, the cosets of $\langle a \rangle$ are black
vertices, the cosets of $\langle b\rangle$ are white vertices,
and the local rotations at black and white vertices,
respectively, correspond to the multiplication by $a$ and $b$.

In Section~3 we introduce the concept of reciprocal skew-morphisms,
and prove the main result,
Theorem~\ref{thm:corresp-all}, which establishes the claimed
correspondence between complete regular dessins and
pairs of reciprocal skew-morphisms of the cyclic groups.

An important part of the classification of complete regular
dessins consists of identifying all pairs of integers $m$ and
$n$ for which there exists a unique complete regular dessin up
to isomorphism and interchange of colours. This problem will be
discussed in Section~4. By Theorem~\ref{thm:corresp-all}, we
ask for which integers $m$ and $n$ the only reciprocal pair of
skew-morphisms is the trivial pair formed by the two identity
automorphisms. Yet in other words, we wish to determine all
pairs of integers $m$ and $n$ that give rise to only one exact
product of cyclic groups $\mathbb{Z}_m$ and $\mathbb{Z}_m$,
which necessarily must be the direct product
$\mathbb{Z}_m\times \mathbb{Z}_n$. The answer is given in
Theorem~\ref{the:uniq} which states that all this occurs
precisely when $\gcd(m,\phi(n))=\gcd(\phi(m),n)=1$, where
$\phi$ is the Euler's totient function. This theorem presents
six equivalent conditions one of which corresponds to a recent
result of Fan and Li \cite{FL2017} about the existence of a
unique edge-transitive orientable embedding of a complete
bipartite graph. While the proof in \cite{FL2017} is based on
the structure of exact bicyclic groups, our proof employs the
correspondence theorems established in Section~3.
Theorem~\ref{the:uniq} is a direct generalisation of a result
of Jones, Nedela, and \v Skoviera \cite{JNS2008} where it is
assumed that the complete dessin in question admits an
external symmetry swapping the two partition sets.
Theorem~\ref{the:uniq} generalises the main result of
\cite{FL2017} also by extending it to all products of cyclic
groups rather than just to those where the intersection of
factors is trivial. In particular, we prove that every group
that factorises as a product of two cyclic subgroups of orders $m$ and $n$
is abelian if and only if
$\gcd(m,\phi(n))=\gcd(\phi(m),n)=1$, where $\phi$ is Euler's
totient function. This generalizes an old result due to Burnside
which states that every group of order $n$ is cyclic if and only if
$\gcd(n,\phi(n))=1$~\cite[\S10.1]{Robinson1996}

Finally, in Section~5 we deal with the symmetric case, that is,
the reciprocal skew-morphism pairs of the form
$(\varphi,\varphi)$. In this case, our results are related to
 the classification of orientably regular embeddings of the
complete bipartite graphs $K_{n,n}$ just recently obtained
in a series of papers~\cite{DJKNS2007, DJKNS2010,DJKNS2013, Jones2010,JNS2007,JNS2008, NSZ}.
These maps correspond to the
complete regular dessins admitting an additional external
symmetry swapping the two partition sets.

\section{Complete regular dessins}
It is well known that every dessin, as defined in the previous
section, can be regarded as a two-generator transitive
permutation group acting on a non-empty finite set
\cite{JS1996}. Given a dessin $\mathcal{D}$ on an oriented
surface $C$, we can define two permutations $\rho$ and
$\lambda$ on the edge set of $\mathcal{D}$ as follows: For every
 black vertex $v$ and every white vertex $w$ let $\rho_v$
and $\lambda_w$ be the cyclic permutations of edges incident
with $v$ or $w$, respectively, induced by the orientation of
$C$. Set $\rho=\prod_v \rho_v$ and $\lambda=\prod_w \lambda_w$,
where $v$ and $w$ run through the set of all black and white
vertices, respectively. Since the underlying graph of
$\mathcal{D}$ is connected, the group
$G=\langle\rho,\lambda\rangle$ is transitive. Conversely, given
a transitive permutation group $G=\langle\rho,\lambda\rangle$
acting on a finite set $\Omega$, a dessin $\mathcal{D}$ can be
constructed as follows: Take $\Omega$ to be the edge set of
$\mathcal{D}$, the orbits of $\rho$ to be the black vertices,
and the orbits of $\lambda$ to the white vertices, with
incidence being defined by containment. The vertices and edges
of $\mathcal{D}$ clearly form a bipartite graph $\Gamma$, the
\textit{underlying graph} of $\mathcal{D}$. The underlying
graph is connected, because the action of $G$ on $\Omega$ is
transitive. The cycles of $\rho$ and $\lambda$ determine the
local rotations around black and white vertices, respectively,
thereby giving rise to a $2$-cell embedding of $\Gamma$ into an
oriented surface. Summing up, we can identify a dessin with a
triple $(\Omega;\rho,\lambda)$ where $\Omega$ is a nonempty
finite set, and $\rho$ and $\lambda$ are permutations of
$\Omega$ such that the group $\langle\rho,\lambda\rangle$ is
transitive on $\Omega$; this group is called the
\textit{monodromy group} of $\mathcal{D}$ and is denoted by
$\Mon(\mathcal{D})$.

Two dessins $\mathcal{D}_1=(\Omega_1;\rho_1,\lambda_1)$ and
$\mathcal{D}_2=(\Omega_2;\rho_2,\lambda_2)$ are
\textit{isomorphic} provided that there is a bijection
$\alpha\colon\Omega_1\to\Omega_2$ such that
$\alpha\rho_1=\rho_2\alpha$ and
$\alpha\lambda_1=\lambda_2\alpha$. An isomorphism of a dessin
$\mathcal{D}$ to itself is an \textit{automorphism} of
$\mathcal{D}$. It follows that the automorphism group
$\Aut(\mathcal{D})$ of $\mathcal{D}$ is the centraliser of
$\Mon(\mathcal{D})=\langle \rho,\lambda\rangle$ in the
symmetric group $\mathrm{Sym}(\Omega)$. As $\Mon(\mathcal{D})$
is transitive, $\Aut(\mathcal{D})$ is semi-regular on $\Omega$.
If $\Aut(\mathcal{D})$ is transitive, and hence regular on
$\Omega$, the dessin $\mathcal{D}$ itself is called
\textit{regular}.

Since every regular action of a group on a set is equivalent to
its action on itself by multiplication, every regular dessin
can be identified with a triple $\mathcal{D}=(G; a,b)$ where
$G$ is a finite group generated by two elements $a$ and $b$.
Given such a triple $\mathcal{D}=(G; a,b)$, we can define the
edges of $\mathcal{D}$ to be the elements of $G$, the black
vertices to be the left cosets of the cyclic subgroup $\langle
a \rangle$, and the white vertices to be the left cosets of the
cyclic subgroup $\langle b\rangle$. An edge $g\in G$ joins the
vertices $s\langle a\rangle$ and $t\langle b\rangle$ if and
only if $g\in s\langle a\rangle\cap t\langle b\rangle$. In
particular, the underlying graph is simple if and only if
$\langle a\rangle\cap \langle b\rangle=1$. The local rotation
of edges around a white vertex $s\langle a\rangle$ corresponds
to the right translation by the generator $a$, that is,
$sa^i\mapsto sa^{i+1}$ for any integer $i$. Similarly, the
local rotation of edges around a black vertex $t\langle
b\rangle$ corresponds to the right translation by the generator
$b$, that is, $tb^i\mapsto tb^{i+1}$ for any integer $i$. It
follows that $\Mon(\mathcal{D})$ can be identified with the
group of all right translations of $G$ by the elements of $G$
while $\Aut(\mathcal{D})$ can be identified with the group of
all left translations of $G$ by the elements of $G$. In
particular, $\Mon(\mathcal{D})\cong\Aut(\mathcal{D})\cong G$
for every regular dessin $\mathcal{D}$.

It is easy to see that two regular dessins
$\mathcal{D}_1=(G_1;a_1,b_1)$ and $\mathcal{D}_2=(G_2;a_2,b_2)$
are isomorphic if and only if the triples $(G_1;a_1,b_1)$ and
$(G_2;a_2,b_2)$ are \textit{equivalent}, that is, whenever
there is a group isomorphism $G_1\to G_2$ such that $a_1\mapsto
a_2$ and $b_1\mapsto b_2$. Consequently, for a given
two-generator group $G$, the isomorphism classes of regular
dessins $\mathcal{D}$ with $\Aut(\mathcal{D})\cong G$ are in
one-to-one correspondence with the orbits of the action of
$\Aut(G)$ on the generating pairs $(a,b)$ of $G$.

Following Lando and Zvonkin~\cite{LZ2004}, for a regular dessin
$\mathcal{D}=(G; a,b)$ we define its \textit{reciprocal dessin}
to be the regular dessin $\mathcal{D}^*=(G; b,a)$.
Topologically, $\mathcal{D}^*$ arises from $\mathcal{D}$ simply
by interchanging the vertex colours of $\mathcal{D}$. Thus the
reciprocal dessin has the same underlying graph, the same
supporting surface, and the same automorphism group as the
original one. Clearly, $\mathcal{D}^*$ is isomorphic to
$\mathcal{D}$ if and only if $G$ has an automorphism swapping
the generators $a$ and $b$. If this occurs, the regular dessin
$\mathcal{D}$ will be called \textit{symmetric}. A symmetric
dessin possesses an external symmetry which transposes the
vertex-colours and thus is essentially the same thing as an
orientably regular bipartite map.

In this paper we apply the general theory to regular dessins
whose underlying graph is a complete bipartite graph. A regular
dessin $\mathcal{D}$ will be called an \textit{$(m,n)$-complete
regular dessin}, or simply a \textit{complete regular dessin},
if its underlying graph is the complete bipartite graph
$K_{m,n}$ whose $m$-valent vertices are coloured black and
$n$-valent vertices are coloured white. If $\mathcal{D}$ is an
$(m,n)$-complete regular dessin, then the reciprocal dessin
$\mathcal{D}^*$ is an $(n,m)$-complete regular dessin. Thus all
complete regular dessins appear in reciprocal pairs. Note that
$m=n$ does not necessarily imply that the dessin is symmetric.

Complete regular dessins can be easily described in group
theoretical terms: their automorphism group is just an exact
bicyclic group. This fact was first observed by Jones et al. in
\cite{JNS2008}. A bicyclic group
$G=\langle a\rangle\langle b\rangle$ with $|a|=m$ and $|b|=n$
will be called an
\textit{$(m,n)$-bicyclic group} and $(G;a,b)$ an
\textit{$(m,n)$-bicyclic triple}. Note that an
exact $(m,n)$-bicyclic group has precisely $mn$ elements.

The following result was proved in \cite{JNS2008}.
We include the proof for the reader's convenience.

\begin{theorem}\label{CORR0} %{\rm (\cite{JNS2008})}
A regular dessin $\mathcal{D}=(G;a,b)$ is complete if and only if
$G=\langle a\rangle\langle b\rangle$ is an exact bicyclic
group. Furthermore,  the isomorphism classes of
$(m,n)$-complete regular dessins are in a one-to-one
correspondence with the equivalence classes of exact
$(m,n)$-bicyclic triples.
\end{theorem}

\begin{proof}
Let $\mathcal{D}=(G;a,b)$ be a regular dessin with
underlying graph $K_{m,n}$. Then $|G|=mn$, $|a|=m$, and
$|b|=n$. Since the underlying graph is simple, $\langle
a\rangle\cap\langle b\rangle=1$. It follows that $|\langle
a\rangle\langle b\rangle|=|\langle a\rangle||\langle
b\rangle|/|\langle a\rangle\cap\langle b\rangle|=mn$, whence
$G=\langle a\rangle\langle b\rangle$. Therefore, $(G;a,b)$ is
an exact $(m,n)$-bicyclic triple.

For the converse, let $(G;a,b)$ be an exact $(m,n)$-bicyclic
triple. The general theory explained above implies that the
underlying graph is connected, with $mn$ edges, the black
vertices being $m$-valent and white vertices being $n$-valent.
Since $\langle a\rangle\cap\langle b\rangle=1$, the underlying
graph is simple, and hence it must be the complete bipartite
graph $K_{m,n}$.

Finally, two $(m,n)$-complete regular dessins
$\mathcal{D}_1=(G_1;a_1,b_1)$ and $\mathcal{D}_1=(G_2;a_2,b_2)$
are isomorphic if and only if there is an isomorphism from
$G_1$ onto $G_2$ taking $a_1\mapsto a_2$ and $b_1\mapsto b_2$,
that is, if and only if the triples $(G_1;a_1,b_1)$ and
$(G_2;a_2,b_2)$ are equivalent.
\end{proof}

\begin{example}\label{EXP1}
For each pair of positive integers $m$ and $n$ there is an exact
bicyclic triple $(G;a,b)$ where
\[
G=\langle a,b\mid a^m=b^n=[a,b]=1\rangle
=\langle a\rangle\times \langle b\rangle
 \cong\mathbb{Z}_m\times\mathbb{Z}_n,
 \]
with $[a,b]$ denoting the commutator $a^{-1}b^{-1}ab$. It is
easy to see that this triple is uniquely determined by the
group $\mathbb{Z}_m\times\mathbb{Z}_n$ up to order of
generators and equivalence, so up to reciprocality this group
gives rise to a unique complete regular dessin with underlying
graph $K_{m,n}$. We call this dessin the \textit{standard}
$(m,n)$-complete dessin. If $m=n$, the group $G$ has an
automorphism transposing $a$ and $b$, which implies that in
this case the dessin is symmetric. The corresponding
embedding is the \textit{standard embedding} of $K_{n,n}$
described in ~\cite[Example~1]{JNS2008}. The associated
algebraic curves coincide with the Fermat curves.
\end{example}

\section{Reciprocal skew-morphisms}

In this section we establish a correspondence between exact
bicyclic triples and certain pairs of skew-morphisms of cyclic
groups.

Recall that a \textit{skew-morphism} $\varphi$ of a
finite group $A$ is a bijection $A\to A$ fixing the identity of
$A$ for which there exists an associated power function
$\pi\colon A\to\mathbb{Z}$ such that
\[
\varphi(xy)=\varphi(x)\varphi^{\pi(x)}(y)
\]
for all $x,y\in A$. It may be useful to realise that $\pi$ is
not uniquely determined by~$\varphi$. However, if $\varphi$ has
order $m$, then $\pi$ can be regarded as a function
$A\to\mathbb{Z}_m$, which is unique. In the special case where
 $\pi(x)=1$ for all $x\in A$, $\varphi$ is a group
automorphism; conversely, every group automorphism is a
skew-morphism with power function being the constant function
$\pi(x)=1$.

Skew-morphisms have a number of important properties, sometimes
very different from those of group automorphisms. In our
treatment we restrict ourselves to a few basic properties of
skew-morphisms needed in this paper. For a deeper account
we refer the reader to \cite{CJT2016, JS2002, KN2011, HWYZ2017,
ZD2016}.

The next three properties of skew-morphisms are well known and
were proved in \cite{JN2015,JS2002,HWYZ2017}. For convenience,
we include the proof of the third of them.

\begin{lemma}\label{SKEW}
Let $\varphi$ be a skew-morphism of a finite group $A$ with
associated power function~$\pi$.  Let $n$ be the order of
$\varphi$. Then:
\begin{itemize}
\item[\rm(i)] for any two elements $x,y\in A$ and an
    arbitrary positive integer $k$ one has
$$\varphi^k(xy)=\varphi^k(x)\varphi^{\sigma(x,k)}(y) \quad
   \text{ where }\quad\sigma(x,k)=\sum\limits_{i=1}^k\pi(\varphi^{i-1}(x));$$

\item[\rm(ii)] for every element $x\in A$ one has
    $\mathcal{O}_{x^{-1}}=\mathcal{O}_x^{-1}$, where
    $\mathcal{O}_x$ denotes the orbit of $\varphi$ containing $x$;

\item[\rm(iii)] for every $x\in A$ one has $\sigma(x,n)\equiv0\pmod{n}$.
\end{itemize}
\end{lemma}

\begin{proof}
We prove (iii). For any element $x\in A$, let
$m=|\mathcal{O}_x|$ be the length of the orbit of $\varphi$
containing $x$. Then $n=mk$ for some integer $k$. By (i) we
have
\[
1=\varphi^m(xx^{-1})=\varphi^m(x)\varphi^{\sigma(x,m)}(x^{-1})
=x\varphi^{\sigma(x,m)}(x^{-1}).
\]
By (ii), $|\mathcal{O}_{x^{-1}}|=|\mathcal{O}_x|=m$,
so $\sigma(x,m)\equiv0\pmod{m}$. If follows that
\[
\sigma(x,n)=\sum_{i=1}^n\pi(\varphi^{i-1}(x))
=k\sum_{i=1}^m\pi(\varphi^{i-1}(x))
=k\sigma(x,m)\equiv0\pmod{n},
\]
as required.
\end{proof}

Let $G$ be a finite group which is expressible as a product
$AB$ of two subgroups $A$ and $B$ where $B$ is cyclic and
$A\cap B=1$; in this situation we say that $B$ is a
\textit{cyclic complement} of $A$.  Choose a generator $b$ of
$B$. Since $G=AB=BA$, for every element $x\in A$ we can write
the product $bx$ in the form $yb^k$, so
$$bx=yb^k$$
for some $y\in A$ and $k\in\mathbb{Z}_{|b|}$. Observe that both
$y\in A$ and $k\in\mathbb{Z}_{|b|}$ are uniquely determined by
$x$. Therefore we can define the functions $\varphi\colon A\to
A$ and $\pi:A\to\mathbb{Z}_{|b|}$ by setting
\begin{align}\label{eq:induced}
\varphi(x)=y\quad\text{and}\quad \pi(x)=k.
\end{align}
It is not difficult to verify that $\varphi$ is a skew-morphism
of $A$ and $\pi$ is an associated power function (see, for
example \cite[p.~73]{CJT2016}). We call $\varphi$ the
skew-morphism \textit{induced} by $b$. The order $|\varphi|$ of
this skew-morphism was determined in \cite[Lemma~4.1]{CJT2016}
and equals the index $|\langle b\rangle:\langle b\rangle_G|$
where $\langle b\rangle_G=\cap_{g\in G}\langle b\rangle^g$. It
follows that the power function $\pi$ can be further reduced to
a function $A\to\mathbb{Z}_{|\varphi|}$, still denoted by
$\pi$.

We now focus on the case where the subgroup $A$ is also cyclic,
which means that $G$ is an exact bicyclic group. The
subgroups $A$ and $B$ can now be taken as cyclic complements of
each other. Therefore every generator $a$ of $A$ induces a
skew-morphism of $B$ and every generator $b$ of $B$ induces a
skew-morphism of $A$. In other words, every exact bicyclic
triple $(G;a,b)$ gives rise to a pair of skew-morphisms, one
for each of the two cyclic subgroups.

We now describe these two skew-morphisms more precisely. Let
$G=\langle a\rangle\langle b\rangle$ be an exact bicyclic group
where $\langle a\rangle\cong\mathbb{Z}_m$ and $\langle
b\rangle\cong\mathbb{Z}_n$. Then for every $x\in\mathbb{Z}_n$
there exist unique elements $i\in\mathbb{Z}_n$ and
$j\in\mathbb{Z}_m$ such that $ab^x=b^ia^j$.
In accordance with (\ref{eq:induced}) we define the functions
$\varphi\colon\mathbb{Z}_n\to\mathbb{Z}_n$ and
$\pi\colon\mathbb{Z}_n\to\mathbb{Z}_m$ by
\begin{align}\label{REC1}
\varphi(x)=i\quad\text{and}\quad \pi(x)=j.
\end{align}
Similarly, each $y\in\mathbb{Z}_m$ determines a unique pair of
elements $k\in\mathbb{Z}_m$ and $l\in\mathbb{Z}_n$  such that
$ba^y=a^kb^l$. Define the functions
${\varphi^*}\colon\mathbb{Z}_m\to\mathbb{Z}_m$ and
${\pi^*}\colon\mathbb{Z}_m\to\mathbb{Z}_n$  to be
\begin{align}\label{REC2}
{\varphi^*}(y)=k\quad\text{and}\quad{\pi^*}(y)=l.
\end{align}

Next we show that the mappings $\varphi$ and ${\varphi^*}$
defined by (\ref{REC1}) and (\ref{REC2}) are indeed
skew-morphisms of $\mathbb{Z}_n$ and $\mathbb{Z}_m$ with power
functions $\pi$ and ${\pi^*}$, respectively. We will further
show that this pair of skew-morphisms can be characterised by
two simple properties. For this purpose, we need the following
definition.

\begin{definition}\label{def:recipr}
A pair $(\varphi,{\varphi^*})$ of skew-morphisms
$\varphi\colon\mathbb{Z}_n\to\mathbb{Z}_n$ and
${\varphi^*}\colon\mathbb{Z}_m\to\mathbb{Z}_m$ with power
functions $\pi$ and $\pi^*$, respectively, will be called
\textit{$(m,n)$-reciprocal} if the following two conditions are
satisfied:
\begin{itemize}
\item[\rm(i)] the orders of $\varphi$ and ${\varphi^*}$ divide
    $m$ and $n$, respectively,
\item[\rm(ii)] $\pi(x)=-{\varphi^*}^{-x}(-1)$ and
    ${\pi^*}(y)=-\varphi^{-y}(-1)$ are power functions for
    $\varphi$ and ${\varphi^*}$, respectively.
\end{itemize}
\end{definition}

If $m=n$ and $(\varphi,{\varphi^*})$ is an $(n,n)$-reciprocal
pair of skew-morphisms it may, but need not, happen that
$\varphi={\varphi^*}$. If it does, then the pair
$(\varphi,\varphi)$, as well as the skew-morphism $\varphi$
itself, will be called \textit{symmetric}. Note that a
skew-morphism $\varphi$ of $\mathbb{Z}_n$ is symmetric if and
only if $|\varphi|$ divides $n$ and $\pi(x)=-\varphi^{-x}(-1)$
is a power function of $\varphi$.

\begin{proposition}\label{prop:bic2recip}
Let $(G;a,b)$ be an exact $(m,n)$-bicyclic triple, and let
$\varphi:\mathbb{Z}_n\to\mathbb{Z}_n$ and
${\varphi^*}:\mathbb{Z}_m\to\mathbb{Z}_m$ be mappings defined
by \eqref{REC1} and \eqref{REC2}, respectively. Then
$(\varphi,{\varphi^*})$ is an $(m,n)$-reciprocal pair of
skew-morphisms. In particular, if $G$ has an automorphism
transposing $a$ and $b$, then $\varphi={\varphi^*}$ and the
pair is symmetric.
\end{proposition}
\begin{proof}
We first show that $\varphi$ is a permutation of $\mathbb{Z}_n$
that fixes the identity element $0$. For any two elements
$x_1,x_2\in\mathbb{Z}_n$ we have
$ab^{x_1}=b^{\varphi(x_1)}a^{\pi(x_1)}$ and
$ab^{x_2}=b^{\varphi(x_2)}a^{\pi(x_2)}$, or equivalently
$b^{x_1}a^{-\pi(x_1)}=a^{-1}b^{\varphi(x_1)}$ and
$b^{x_2}a^{-\pi(x_2)}=a^{-1}b^{\varphi(x_2)}$. If
$\varphi(x_1)=\varphi(x_2)$, then
$b^{x_1}a^{-\pi(x_1)}=b^{x_2}a^{-\pi(x_2)}$, so
$b^{x_1-x_2}=a^{\pi(x_1)-\pi(x_2)}$. Since $\langle
a\rangle\cap\langle b\rangle=1$, we get $x_1= x_2$, which means that
$\varphi$ is a bijection $\mathbb{Z}_n\to \mathbb{Z}_n$. Note that
$a=ab^0=b^{\varphi(0)}a^{\pi(0)}$, so
$b^{\varphi(0)}=a^{1-\pi(0)}\in\langle a\rangle\cap\langle
b\rangle$. Thus $\varphi(0)=0$.

From the commuting rules $ab^{x_1}=b^{\varphi(x_1)}a^{\pi(x_1)}$ and
$ab^{x_2}=b^{\varphi(x_2)}a^{\pi(x_2)}$ we derive,
for all $x_1,x_2\in\mathbb{Z}_n$, that
 \[
 b^{\varphi(x_1+x_2)}a^{\pi(x_1+x_2)}=ab^{x_1+x_2}
 =b^{\varphi(x_1)}a^{\pi(x_1)}b^{x_2}
 =b^{\varphi(x_1)+\varphi^{\pi(x_1)}(x_2)}a^{\sigma(x_2,\pi(x_1))},
 \]
where
$\sigma(x_2,\pi(x_1))=\sum\limits_{i=1}^{\pi(x_1)}\pi(\varphi^{i-1}(x_2)).$
Consequently
$\varphi(x_1+x_2)=\varphi(x_1)+\varphi^{\pi(x_1)}(x_2)$, and
therefore $\varphi$ is a skew-morphism of $\mathbb{Z}_n$ with
$\pi$ as a power function. The proof that ${\varphi^*}$ is a
skew-morphism of $\mathbb{Z}_m$ with ${\pi^*}$ as a power
function is similar and therefore omitted.

As mentioned above, in~\cite[Lemma 4.1]{CJT2016} it was proved
that the orders of $\varphi$ and ${\varphi^*}$ coincide with
the indices $|\langle a\rangle:\bigcap_{g\in G}\langle
a\rangle^g|$ and $|\langle b\rangle:\bigcap_{g\in G}\langle
b\rangle^g|$. Hence $|\varphi|$ and
$|{\varphi^*}|$ divide $|\langle a\rangle|=m$ and
$|\langle b\rangle|=n$, respectively.
If we apply induction to the equations $ab^x=b^{\varphi(x)}a^{\pi(x)}$
and $ba^{y}=a^{{\varphi^*}(y)}b^{{\pi^*}(y)}$ we get
\[
a^kb^x=b^{\varphi^k(x)}a^{\sigma(x,k)}\quad\text{and}\quad b^la^y
      =a^{{\varphi^*}^l(y)}b^{{\sigma^*}(y,l))},
\]
 where
\[
\sigma(x,k)=\sum\limits_{i=1}^k\pi(\varphi^{i-1}(x))\quad\text{and}
\quad {\sigma^*}(y,l)=\sum\limits_{i=1}^l{\pi^*}({\varphi^*}^{i-1}(y)).
\]
By inverting these identities we obtain
\begin{align}\label{IDDD}
b^{-x}a^{-k}=a^{-\sigma(x,k)}b^{-\varphi^k(x)}\quad\text{and}
\quad a^{-y}b^{-l}=b^{-{\sigma^*}(y,l)}a^{-{\varphi^*}^l(y)}.
\end{align}
The first equation of \eqref{IDDD} with $x=-1$ and $k=-y$ yields
$ba^y=a^{-\sigma(-1,-y)}b^{-\varphi^{-y}(-1)}$, which we compare
with the rule $ba^y=a^{{\varphi^*}(y)}b^{{\pi^*}(y)}$ and get
\[
a^{{\varphi^*}(y)}b^{{\pi^*}(y)}=a^{-\sigma(-1,-y)}b^{-\varphi^{-y}(-1)}.
 \]
Consequently ${\pi^*}(y)=-\varphi^{-y}(-1)$.  Similarly, inserting
$y=-1$ and $l=-x$ to the second equation of \eqref{IDDD} we get
$ab^x=b^{-{\sigma^*}(-1,-x)}a^{-{\varphi^*}^{-x}(-1)}$, and combining
this with the rule $ab^x=b^{\varphi(x)}a^{\pi(x)}$ we derive
$\pi(x)=-{\varphi^*}^{-x}(-1)$. Hence, the pair
$(\varphi,{\varphi^*})$ is $(m,n)$-reciprocal.

Finally, if $G$ has an automorphism $\theta$ transposing $a$
and $b$, then clearly $m=n$. By applying $\theta$ to the identity
$ba^x=a^{{\varphi^*}(x)}b^{{\pi^*}(x)}$ we obtain
$ab^x=\theta(ba^x)=\theta(a^{{\varphi^*}(x)}b^{{\pi^*}(x)})
=b^{{\varphi^*}(x)}a^{{\pi^*}(x)}$. If we compare the last
identity with the rule $ab^x=b^{\varphi(x)}a^{\pi(x)}$ we
obtain ${\varphi^*}=\varphi$, which means that $\varphi$ is a
symmetric skew-morphism of $\mathbb{Z}_n$, as required.
\end{proof}

We have just shown that every exact $(m,n)$-bicyclic triple
determines an $(m,n)$-reciprocal pair of skew-morphisms. Our
next aim is to show that the converse is also true. Let
$(\varphi,{\varphi^*})$ be an $(m,n)$-reciprocal pair of
skew-morphisms of $\mathbb{Z}_n$ and $\mathbb{Z}_m$ with power
functions $\pi$ and ${\pi^*}$, respectively. For the sake of
clarity we relabel the elements of $\mathbb{Z}_n$ and
$\mathbb{Z}_m$ by setting
\[
\mathbb{Z}_n=\{0,1,\cdots,(n-1)\}\quad\text{and}\quad \mathbb{Z}_m
            =\{0',1',\cdots,(m-1)'\},
\]
so that $\mathbb{Z}_n\cap \mathbb{Z}_m=\emptyset$. Let
\[
\rho=(0,1,\cdots,(n-1))\quad\text{and}\quad{\rho^*}=(0',1',\cdots,(m-1)')
 \]
denote the \textit{cyclic shifts} in $\mathbb{Z}_n$ and
$\mathbb{Z}_m$, respectively. We now extend the permutations
$\varphi$, $\rho$, ${\varphi^*}$, and ${\rho^*}$ to the set
$\mathbb{Z}_n\cup\mathbb{Z}_m$ in a natural way, and define a
permutation group acting on the set
$\mathbb{Z}_m\cup\mathbb{Z}_n$ by
%\begin{align}
\[
G=\langle a,b\rangle,\quad \text{where}\quad a
=\varphi{\rho^*} \quad \text{and}
\quad b={\varphi^*}\rho.
\]
%end\{align}
If we regard $\mathbb{Z}_m\cup\mathbb{Z}_m$ as the vertex set
of the complete bipartite graph $K_{m,n}$ with natural
bipartition, it becomes obvious that $G\le\Aut(K_{m,n})$. The
following result shows that $G$ is in fact the automorphism
group of an $(m,n)$-complete regular dessin.

\begin{proposition}\label{prop:recip2bic}
Given an $(m,n)$-reciprocal pair of skew-morphisms
$(\varphi,{\varphi^*})$, the triple $(G;a,b)$, where
$a=\varphi{\rho^*}$ and $b={\varphi^*}\rho$ are permutations
acting on the disjoint union $\mathbb{Z}_m\cup\mathbb{Z}_m$, is
an exact $(m,n)$-bicyclic triple. Furthermore, the given
skew-morphisms  $\varphi$ and ${\varphi^*}$ coincide with the
skew-morphisms induced by the generators $a$ and $b$ in this
triple.
\end{proposition}

\begin{proof}
Let $(\varphi,{\varphi^*})$ be an $(m,n)$-reciprocal pair of
skew-morphisms. The definition of reciprocality requires
$|\varphi|$ to divide $m$ and $|{\varphi^*}|$ to divide $n$.
Since $a=\varphi{\rho^*}$ and $b={\varphi^*}\rho$, where $\rho$
and ${\rho^*}$ are cyclic shifts in $\mathbb{Z}_n$ and
$\mathbb{Z}_m$, respectively, it follows that $|a|=m$ and
$|b|=n$. Further, if $x\in\langle a\rangle\cap\langle
b\rangle$, then $a^i=x=b^j$ for some integers $i$ and $j$, so
$(\varphi{\rho^*})^i=({\varphi^*}\rho)^j$. Since
$\varphi,\rho\in\mathrm{Sym}(\mathbb{Z}_n)$,
${\varphi^*},{\rho^*}\in\mathrm{Sym}(\mathbb{Z}_m)$, and
$\mathbb{Z}_m\cap\mathbb{Z}_n=\emptyset$, we have
$[\varphi,{\rho^*}]=1$ and $[{\varphi^*},\rho]=1$, therefore
$\varphi^i{\rho^*}^i=\rho^j{\varphi^*}^j$, whence
$\varphi^i=\rho^j$ and ${\rho^*}^i={\varphi^*}^j$. Since
$\varphi(0)=0$ and $\rho$ is a full cycle, we have $n\mid j$.
Similarly we obtain $m\mid i$. Hence $x=1$, and therefore
$\langle a\rangle\cap\langle b\rangle=1$.

Next we show that $\langle a\rangle\langle b\rangle$ is a
subgroup of $G$. It is sufficient to verify that $\langle
a\rangle\langle b\rangle=\langle b\rangle\langle a\rangle$. For
this purpose we need to show that for all $x\in\mathbb{Z}_n$
and $y\in\mathbb{Z}_m$ there exist numbers $\alpha(x)$,
$\beta(x)$, ${\alpha^*}(y)$ and ${\beta^*}(y)$ such that the
following commuting rules hold:
\begin{align}\label{EQNN1A}
ab^x=b^{\alpha(x)}a^{\beta(x)}\quad\text{and}\quad ba^y
    =a^{{\alpha^*}(y)}b^{{\beta^*}(y)}.
\end{align}
If we substitute $\varphi{\rho^*}$ and ${\varphi^*}\rho$ for $a$ and $b$
we can see that the equations in \eqref{EQNN1A} are equivalent to
the following four equations:
\begin{align}
&\varphi\rho^x=\rho^{\alpha(x)}\varphi^{\beta(x)},&&
{\rho^*}{\varphi^*}^x={\varphi^*}^{\alpha(x)}{\rho^*}^{\beta(x)};
\label{EQNN1}\\
&{\varphi^*}{\rho^*}^y={\rho^*}^{{\alpha^*}(y)}{\varphi^*}^{{\beta^*}(y)},
&&\rho\varphi^y=\varphi^{{\alpha^*}(y)}\rho^{{\beta^*}(y)}.\label{EQNN2}
\end{align}
Since $\varphi$ and ${\varphi^*}$ are skew-morphisms with $\pi$
and ${\pi^*}$ as power functions, for all $i\in \mathbb{Z}_n$
and $j\in\mathbb{Z}_m$ we have
\begin{align*}
\varphi\rho^x(i)&=\varphi(x+i)=\varphi(x)+\varphi^{\pi(x)}(i)
                 =\rho^{\varphi(x)}\varphi^{\pi(x)}(i);\\
{\varphi^*}{\rho^*}^y(j)&={\varphi^*}(y+j)={\varphi^*}(y)+{\varphi^*}^{{\pi^*}(y)}(j)
                   ={\rho^*}^{{\varphi^*}(y)}{\varphi^*}^{{\pi^*}(y)}(j).
\end{align*}
These equations imply that the first equations in \eqref{EQNN1}
and \eqref{EQNN2} hold if we set $\alpha(x)=\varphi(x)$,
$\beta(x)=\pi(x)$, ${\alpha^*}(y)={\varphi^*}(y)$ and
${\beta^*}(y)={\pi^*}(y)$.

Employing induction, from the first equations in \eqref{EQNN1}
and \eqref{EQNN2} we derive that
 \[
  \varphi^k\rho^u=\rho^{\alpha^k(u)}\varphi^{\tau(u,k)}\quad\text{and}
  \quad {\varphi^*}^l{\rho^*}^v={\rho^*}^{{\alpha^*}^l(v)}{\tau^*}^{{\tau^*}(v,l)},
 \]
where
 \[
 \tau(u,k)=\sum_{i=1}^{k}\beta(\alpha^{i-1}(u))\quad\text{and}
 \quad  {\tau^*}(v,l)=\sum_{i=1}^{l}{\beta^*}({\alpha^*}^{i-1}(v)).
 \]
By inverting the identities we obtain
 \[
  \rho^{-u}\varphi^{-k}=\varphi^{-\tau(u,k)}\rho^{-\alpha^k(u)}
  \quad\text{and}\quad
  {\rho^*}^{-v}{\varphi^*}^{-l}={\varphi^*}^{-{\tau^*}(v,l)}{\rho^*}^{-{\alpha^*}^l(v)}.
 \]
In particular,
 \[
  \rho\varphi^{y}=\varphi^{-\tau(-1,-y)}\rho^{-\alpha^{-y}(-1)}
  \quad\text{and}\quad
  {\rho^*}{\varphi^*}^{x}={\varphi^*}^{-{\tau^*}(-1,-x)}{\rho^*}^{-{\alpha^*}^{-x}(-1)}.
 \]
Recall that
\[
\beta(x)=\pi(x)=-{\varphi^*}^{-x}(-1)=-{\alpha^*}^{-x}(-1)\quad\text{and}\quad
{\beta^*}(y)={\pi^*}(y)=-\varphi^{-y}(-1)=-\alpha^{-y}(-1).
\]
Thus the second equations in \eqref{EQNN1} and \eqref{EQNN2}
will hold if
\[
\alpha(x)=\varphi(x)\equiv-{\tau^*}(-1,-x)\pmod{|{\varphi^*}|}
\quad\text{and}\quad {\alpha^*}(y)=
{\varphi^*}(y)\equiv-\tau(-1,-y)\pmod{|\varphi|}.
\]
Indeed, by Lemma~\ref{SKEW}~(iii) we have
${\tau^*}(-1,|{\varphi^*}|)\equiv0\pmod{|{\varphi^*}|}$. Since
\begin{align*}
{\tau^*}(-1,|{\varphi^*}|)&=\sum_{i=1}^{|{\varphi^*}|}
{\beta^*}({\alpha^*}^{i-1}(-1))=\sum_{i=1}^{|{\varphi^*}|}
{\pi^*}({\varphi^*}^{i-1}(-1))\\
&=\sum_{i=1}^{|{\varphi^*}|-x}{\pi^*}({\varphi^*}^{i-1}(-1))+\sum_{i=
|{\varphi^*}|-x+1}^{|{\varphi^*}|}{\pi^*}({\varphi^*}^{i-1}(-1))\\
&={\tau^*}(-1,-x)+\sum_{i=1}^x{\pi^*}({\varphi^*}^{-i}(-1))
\pmod{|{\varphi^*}|},
\end{align*}
we obtain
\[
-{\sigma^*}(-1,-x)\equiv \sum_{i=1}^x{\pi^*}({\varphi^*}^{-i}(-1))\pmod{|{\varphi^*}|}.
\]
On the other hand, since $\varphi$ is a skew-morphism of
$\mathbb{Z}_m$, we have
$\varphi(z-1)=\varphi(z)+\varphi^{\pi(z)}(-1)$ for all
$z\in\mathbb{Z}_n$, so
$\varphi(z-1)-\varphi(z)=\varphi^{\pi(z)}(-1)$. By combining
these identities we obtain
\begin{align*}
\varphi(x)
&=-\big(\varphi(0)-\varphi(x)\big)=-\sum_{i=1}^{x}(\varphi(i-1)-\varphi(i))
=-\sum_{i=1}^x\varphi^{\pi(i)}(-1)\\
&=-\sum_{i=1}^x\varphi^{-{\varphi^*}^{-i}(-1)}(-1)
=\sum_{i=1}^x{\pi^*}({\varphi^*}^{-i}(-1))
\equiv-{\sigma^*}(-1,-x)\pmod{|{\varphi^*}|}.
\end{align*}
Thus we have shown that
\begin{align}\label{sigmaidentity}
\varphi(x)\equiv-{\sigma^*}(-1,-x)\equiv
\sum_{i=1}^x{\pi^*}({\varphi^*}^{-i}(-1))\pmod{|{\varphi^*}|}.
\end{align}
By using similar arguments we can prove that
${\alpha^*}(y)={\varphi^*}(y)\equiv-\sigma(-1,-y)\pmod{|\varphi|}$.
Thus, $\langle a\rangle\langle b\rangle$ is a subgroup of $G$,
as claimed.

Finally, since $G=\langle a,b\rangle$, we have $G=\langle
a\rangle\langle b\rangle$, so $(G;a,b)$ is an exact
$(m,n)$-bicyclic triple. Note that
$ab^x=b^{\alpha(x)}a^{\beta(x)}$ and
$ba^y=a^{{\alpha^*}(y)}b^{{\beta^*}(y)}$ with
$\alpha(x)=\varphi(x)$ and ${\alpha^*}(y)={\varphi^*}(y)$. It
follows that $\varphi$ and ${\varphi^*}$ are precisely the
skew-morphisms induced by $a$ and $b$ in the triple $(G;a,b)$.
\end{proof}

Putting together Theorem~\ref{CORR0},
Proposition~\ref{prop:bic2recip}, and
Proposition~\ref{prop:recip2bic} we obtain a one-to-one
correspondence between $(m,n)$-complete regular dessins, exact
$(m,n)$-bicyclic triples, and $(m,n)$-reciprocal pairs of
skew-morphisms.

\begin{theorem}\label{thm:corresp-all}
There exists a one-to-one correspondence between every pair of
the following three types of objects:
\begin{enumerate}
 \item[{\rm (i)}] isomorphism classes of $(m,n)$-complete regular dessins,
 \item[{\rm (ii)}] equivalence classes of exact $(m,n)$-bicyclic triples, and
 \item[{\rm (iii)}] $(m,n)$-reciprocal pairs of skew-morphisms.
\end{enumerate}
\end{theorem}

\begin{proof}
The correspondence between isomorphism classes of
$(m,n)$-complete regular dessins and equivalence classes of
exact $(m,n)$-bicyclic triples has been established in
Theorem~\ref{CORR0}.  It remains to prove that there is a
one-to-one correspondence between equivalence classes of exact
$(m,n)$-bicyclic triples and $(m,n)$-reciprocal pairs of
skew-morphisms.

By Proposition~\ref{prop:bic2recip}, every exact
$(m,n)$-bicyclic triple $(G;a,b)$ determines a reciprocal pair
$(\varphi,{\varphi^*})$ of skew-morphisms of $\mathbb{Z}_n$ and
$\mathbb{Z}_m$. Conversely, by
Proposition~\ref{prop:recip2bic}, every $(m,n)$-reciprocal pair
$(\varphi,{\varphi^*})$ of skew-morphisms determines an exact
$(m,n)$-bicyclic triple $(G;a,b)$. What remains to prove is
the one-to-one correspondence.

If two $(m,n)$-reciprocal pairs $(\varphi_1,\varphi_1^*)$ and
$(\varphi_2,\varphi_2^*)$ are identical, then clearly so will be
the corresponding $(m,n)$-bicyclic triples. Conversely, let
$(G_1;a_1,b_1)$ and $(G_2;a_2,b_2)$ be two equivalent exact
$(m,n)$-bicyclic triples, and let $(\varphi_1,\varphi_1^*)$ and
$(\varphi_2,\varphi_2^*)$ be the corresponding skew-morphisms.
Since $(G_1;a_1,b_1)$ and $(G_2;a_2,b_2)$ are equivalent, the
assignment $\theta\colon a_1\mapsto a_2,b_1\mapsto b_2$ extends
to an isomorphism from $G_1$ to $G_2$; in particular,
$|a_1|=|a_2|$ and $|b_1|=|b_2|$. Set $m=|a_1|$ and $n=|b_1|$.
Recall that the skew-morphisms $\varphi_1$ and $\varphi_2$
induced by $a_1$ and $a_2$ are determined by the rules
$a_1b_1^x=b_1^{\varphi_1(x)}a_1^{\pi_1(x)}$ and
$a_2b_2^y=b_2^{\varphi_2(y)}a_2^{\pi_2(y)}$ where
$x,y\in\mathbb{Z}_{n}$. If we apply the isomorphism $\theta$ to
the first equation we obtain
$a_2b_2^x=\theta(a_1b_1^x)=\theta(b_1^{\varphi_1(x)}a_1^{\pi_1(x)})
         =b_2^{\varphi_1(x)}a_2^{\pi_1(x)}$,
and combining this with the second equation we get
$b_2^{\varphi_2(x)}a_2^{\pi_2(x)}=b_2^{\varphi_1(x)}a_2^{\pi_1(x)}$.
Thus $\varphi_1=\varphi_2$. Using similar arguments we can get
$\varphi_1^*=\varphi_2^*$. Hence,
$(\varphi_1,\varphi_1^*)=(\varphi_2,\varphi_2^*)$.
\end{proof}

The following corollary makes explicit the identity
\eqref{sigmaidentity}, established in the course of the
previous proof, complete with its reciprocal version.

\begin{corollary}\label{SYM}
If $(\varphi,{\varphi^*})$ is an $(m,n)$-reciprocal pair of
skew-morphisms, then $\varphi$ and ${\varphi^*}$ satisfy the
following conditions:
\[
\varphi(x)=\sum_{i=1}^x{\pi^*}({\varphi^*}^{-i}(-1))\pmod{|{\varphi^*}|}\quad\text{and}
\quad {\varphi^*}(y)=\sum_{i=1}^y\pi(\varphi^{-i}(-1))\pmod{|\varphi|}.
\]
\end{corollary}

Next we offer two examples. The first of them deals with the
standard $(m,n)$-complete dessins.

\begin{example}\label{EXP4}
Let us revisit the group $G=\langle a,b\mid
a^m=b^n=[a,b]=1\rangle\cong\mathbb{Z}_m\times\mathbb{Z}_n$
considered in Example~\ref{EXP1} and determine all reciprocal
pairs of skew-morphisms arising from $G$. Obviously, $G$ gives
rise to only one equivalence class of bicyclic triples, so we
only need to consider the skew-morphisms induced by the triple
$(G;a,b)$. By checking the identities (\ref{REC1}) and
(\ref{REC2}) we immediately see that skew-morphisms induced by
the triple $(G;a,b)$ are the identity automorphisms. Thus the
only reciprocal pair of skew-morphisms arising from the group
$\mathbb{Z}_m\times\mathbb{Z}_n$ is
$(\mathrm{id}_n,\mathrm{id}_m)$, where $\mathrm{id}_k$ denotes
the identity mapping $\mathbb{Z}_k\to \mathbb{Z}_k$. In other
words, for every pair of positive integers $m$ and $n$ there exists
only one complete dessin whose automorphism group is isomorphic
to the direct product $\mathbb{Z}_m\times\mathbb{Z}_n$, the
standard $(m,n)$-complete dessin.
\end{example}

\begin{example}
We present a list of all reciprocal pairs of
skew-morphisms of $\mathbb{Z}_9$ and~$\mathbb{Z}_{27}$.
Altogether there are $27$ of them. Each reciprocal pair
$(\varphi,\varphi^*)$ with $\varphi\colon \mathbb{Z}_9\to
\mathbb{Z}_9$ and $\varphi^*\colon \mathbb{Z}_{27}\to
\mathbb{Z}_{27}$ falls into one of the following two types.

\medskip
\textit{Type I}: \textit{Both $\varphi$ and $\varphi^*$ are
group automorphisms.}

\medskip\noindent
In this case $\varphi(x)\equiv ex\pmod{9}$ and
$\varphi^*(y)\equiv fy\pmod{27}$ for some $e\in \mathbb{Z}_9^*$ and $f\in
\mathbb{Z}_{27}^*$. This can happen only under the following
conditions:
\begin{itemize}
\item[\rm(i)]  $e=1$ and $f\in\{1,4,7,10,13,16,19,22,25$\}, or
\item[\rm(ii)] $e\in\{4,7\}$ and $f\in\{1,10,19\}$.
\end{itemize}
Therefore, there are $9+6=15$ reciprocal pairs of Type I.

\bigskip
\textit{Type II}: \textit{$\varphi$ is a group automorphism but
$\varphi^*$ is not.}

\medskip\noindent
In this case $\varphi(x)\equiv ex\pmod{9}$ and
$\varphi^*(y)\equiv y+3t\sum_{i=1}^y\sigma(s,e^{i-1})\pmod{27}$ where
$e\in\{4,7\}$ and
$\sigma(s,e^{i-1})=\sum_{j=1}^{e^{i-1}}s^{j-1}$ where $(s,t)=(4,1),(7,2),(4,4),(7,5),(4,7)$ or $(7,8).$ These give rise to $2\times 6=12$ pairs of reciprocal skew-morphism of
Type II.

\medskip
Reciprocal pairs of Type~I can be derived directly from
Definition~\ref{def:recipr}. For those of Type~II,
Corollary~\ref{SYM} is helpful. More generally, if both $m$ and
$n$ are powers of an odd prime, complete lists of reciprocal
pairs of skew-morphisms can be compiled with the helps of the
results of Kov{\'a}cs and Nedela \cite{KN2017}.
\end{example}

The correspondence established in Theorem~\ref{thm:corresp-all}
implies that the second condition required in the definition of
an $(m,n)$-reciprocal pair of skew-morphisms (see
Definition~\ref{def:recipr}) can be replaced with a simpler
condition.

\begin{corollary}
A pair $(\varphi,{\varphi^*})$ of skew-morphisms
$\varphi\colon\mathbb{Z}_n\to\mathbb{Z}_n$ and
${\varphi^*}\colon\mathbb{Z}_m\to\mathbb{Z}_m$ with power
functions $\pi$ and ${\pi^*}$, respectively, is reciprocal if
and only if the following two conditions are satisfied:
\begin{itemize}
\item[\rm(i)] $|\varphi|$ divides $m$ and $|{\varphi^*}|$ divides $n$, and

\item[\rm(ii)] $\pi(x)={\varphi^*}^{x}(1)$ and ${\pi^*}(y)=\varphi^{y}(1)$.
\end{itemize}
\end{corollary}

\begin{proof}
It is sufficient to replace the original dessin, represented by an exact
$(m,n)$-bicyclic triple $(G;a,b)$ with its mirror image, for which the
corresponding bicyclic triple is $(G;a^{-1},b^{-1})$, and use
Theorem~\ref{thm:corresp-all}.
\end{proof}

\section{The uniqueness theorem}\label{sec:uniqueness}
We have seen in Example \ref{EXP4} that, for each pair of
positive integers $m$ and $n$ there exists, up to
reciprocality, at least one complete regular dessin with
underlying graph $K_{m,n}$, namely, the standard
$(m,n)$-complete dessin. In this section we determine
all the pairs $(m,n)$ for which a complete regular dessin is unique.

Let us call a pair $(m,n)$ of positive integers $m$ and $n$
\textit{singular} if $\gcd(m,\phi(n))=\gcd(n,\phi(m))=1$. A
positive integer $n$ will be called \textit{singular} if the
pair $(n,n)$ is singular, that is, if $\gcd(n,\phi(n))=1$. We
now show that for each non-singular pair $(m,n)$ of positive
integers there exist a non-abelian exact $(m,n)$-bicyclic group.

\begin{example}\label{EXP2}
Let $m$ and $n$ be positive integers with
$\gcd(n,\phi(m))\not=1$ where $\phi$ is the Euler's totient
function. It is well known that for $x\in \mathbb{Z}_m$ the
assignment $1\mapsto x$ extends to an automorphism of
$\mathbb{Z}_m$ if and only if $\gcd(x,m)=1$, and thus
$|\Aut(\mathbb{Z}_m)|=\phi(m).$ Since $\gcd(n,\phi(m))\neq 1$,
there exists a prime $p$ which divides both $n$ and $\phi(m)$.
Let $r$ denote the multiplicative order of $p$ in
$\mathbb{Z}_m$. Then $r\not\equiv 1\pmod{m}$ and
$r^p\equiv1\pmod{m}$. Define a group $G$ with a presentation
\[
G=\langle a,b\mid a^{m}=b^{n}=1, b^{-1}ab=a^{r}\rangle.
 \]

By H\"older's theorem~\cite[Chapter 7]{Johnson}, $G$ is a well-defined metacyclic group of
order $mn$. Since $r\not\equiv1\pmod{m}$, the group $G$ is non-abelian.
Thus, whenever $\gcd(n,\phi(m))\not=1$, there always
exists at least one non-abelian exact $(m,n)$-bicyclic group $G$.
Similarly, if $\gcd(m,\phi(n))\not=1$, there also
exists a non-abelian exact $(m,n)$-bicyclic group of order
$mn$.

We remark that the argument used here is different
from the one employed in the proof of Lemma~3.1 in \cite{{FL2017}}.
\end{example}

We now apply our theory to proving the following theorem.

\begin{theorem}\label{ABEL}
The following statements are equivalent for every pair of
positive integers $m$ and $n$:
\begin{itemize}
\item[{\rm (i)}] Every product of a cyclic
group of order $m$ with a cyclic group of order $n$ is abelian.
\item[{\rm (ii)}] The pair $(m,n)$ is singular.
\end{itemize}
\end{theorem}
\begin{proof}
If (i) holds, then by virtue of Example~\ref{EXP2} the pair
$(m,n)$ must be singular. For the converse, assume that the pair $(m,n)$
is singular and that $G$ is an $(m,n)$-bicyclic group. We prove
the statement by using induction on the size of $|G|$. By a
result of Huppert~\cite{Huppert1953-1} and Douglas
\cite{Douglas1961} (see also \cite[VI.10.1]{Huppert1967}), $G$
is supersolvable, so for the largest prime factor $p$ of $|G|$
the Sylow $p$-subgroup $P$ of $G$ is normal in (see
$G$~\cite[VI.9.1]{Huppert1967}). By the Schur-Zassenhaus
theorem, $G=P\rtimes Q$, a semidirect product of $P$ by $Q$,
where $Q$ is a subgroup of order $|G/P|$ in $G$. To proceed we
distinguish two cases.

\medskip\noindent
\textbf{Case~1.} \textit{$p$ divides exactly one of $m$ and $n$.}
Without loss of generality we may assume that $p\mid m$ and
$p\nmid n$. Let us write $m$ in the form $m=p^em_1$ where
$p\nmid m_1$. Then the normal subgroup $P$ is contained in the
cyclic factor $A=\langle a\rangle$ of $G$ of order $m$, so
$P=\langle a^{m_1}\rangle$. The generator $b$ of the cyclic
factor $B=\langle b\rangle$ of order $n$ induces an
automorphism $a^{m_1}\mapsto (a^{m_1})^r$ of $P$ by conjugation
$b^{-1}a^{m_1}b=(a^{m_1})^r$ where $r$ is an integer coprime to
$p$. It follows that the multiplicative order $|r|$ of $r$ in
$\mathbb{Z}_{p^e}$ divides $|\Aut(P)|=\phi(p^e)$. On the other
hand, $a^{m_1}=b^{-n}a^{m_1}b^n=(a^{m_1})^{r^n}$, so
$r^n\equiv1\pmod{p^e}$, and hence $|r|$ also divides $n$. But
$\phi(p^e)$ divides $\phi(m)$ and $\gcd(n,\phi(m))=1$, so
$r\equiv1\pmod{p^e}$. Therefore $P$ is contained in the centre
of $G$, and hence $G=P\times Q$, where $Q$ is an
$(m_1,n)$-bicyclic group. It is evident that the pair $(m_1,n)$ is also
singular. By induction, $Q$ is abelian, and therefore $G$ is
abelian.

\medskip\noindent
\textbf{Case~2.} \textit{$p$ divides both $m$ and $n$.}
Since $(m,n)$ is a singular pair, $p^2\nmid m$ and $p^2\nmid
n$. Thus $m=pm_1$ and $n=pn_1$ where $p\nmid m_1$, $p\nmid n_1$
and $\gcd(m_1,p(p-1))=\gcd(n_1,p(p-1))=1$. Since
$|G|=|AB|=|A||B|/|A\cap B|,$ the Sylow $p$-subgroup $P$ of $G$
is of order $p$ or $p^2$. If $p$ divides $|A\cap B|$, then
$|P|=p$ and so $P\leq A\cap B$, which is central in $G$.
Therefore, $G=P\times Q$, where $Q$ is an $(m_1,n_1)$-bicyclic
group, and the result follows by induction. Otherwise, $p\nmid
|A\cap B|$, so $P\cong\mathbb{Z}_p\times \mathbb{Z}_p$. We may
view $P$ as a $2$-dimensional vector space over the Galois
field $\mathbb{F}_p$. Let $\Omega$ be the set of
$1$-dimensional subspaces of $P$. Then $|\Omega|=p+1$ and
$\alpha=\langle a^{m_1}\rangle$ belongs to $\Omega$. Consider
the action of $G$ on $P$ by conjugation. The kernel of this
action is $C_G(P)$, so $\overline{G}=G/C_G(P)\leq {\rm
GL}(2,p)$ where $C_G(P)$ denotes the centraliser of $P$ in $G$.
Now we claim that $\overline{G}=1$.

Suppose to the contrary that $\overline{G}\not=1$. Since
$G=\langle a,b\rangle$, we have $\overline{G}=\langle
\overline{a}^p,\overline{b}^p\rangle$, where
$\overline{a}^p=a^pC_G(P)$ and $\overline{b}^p=b^pC_G(P)$.
Hence at least one of  $\overline{a}^p$ and $\overline{b}^p$ is
not the identity of $\overline{G}$, say $\overline{a}^p\not=1$.
Clearly, $|\overline{a}^p|$ divides $m_1$, the order of $a^p$
in $G$.

Note that $\Omega$ is a complete block system of ${\rm
GL}(2,p)$ on $P$ and the induced action of ${\rm GL}(2,p)$ on
$\Omega$ is transitive. By the Frattini argument, $|{\rm
GL}(2,p)|=(p+1)|{\rm GL}(2,p)_\alpha|$, and hence $|{\rm
GL}(2,p)_\alpha|=p(p-1)^2$ as $|{\rm GL}(2,p)|=p(p+1)(p-1)^2$.
On the other hand, $\overline{a}^p$ fixes $\alpha$ as $a$ fixes
the subspace $\langle a\rangle$, implying that $\overline{a}^p\in
{\rm GL}(2,p)_\alpha$. It follows that $|\overline{a^p}|$ divides
$p(p-1)^2$. Since  $|\overline{a^p}|$ divides $m_1$ and
$\gcd(m_1,p(p-1))=1$, we have $|\overline{a^p}|=1$, which is
impossible because $\overline{a}^p\not=1$. Thus
$\overline{G}=1$, as claimed.

Since $\overline{G}=1$, we have $G=C_G(P)$,  and hence
$G=P\times Q$, where $Q=\langle a^{p}\rangle\langle
b^p\rangle$ is an $(m_1,n_1)$-bicyclic group with the pair $(m_1,n_1)$
being singular. The statement now follows by induction.
\end{proof}

The following result easily follows from Theorem~\ref{ABEL}.

\begin{corollary}\label{exactABEL}
Let $m$ and $n$ be positive integers. Then every
group factorisable as an exact product of cyclic subgroups
of orders $m$ and $n$ is abelian if and only if
the pair $(m,n)$ is singular.
\end{corollary}

We summarise the results of this section in the following
theorem.

\begin{theorem}\label{the:uniq}
The following statements are equivalent for any pair of
positive integers $m$ and $n$:
\begin{itemize}
\item[\rm(i)] The pair $(m,n)$ is singular.

\item[\rm(ii)]Every finite group factorisable as a product
    of two cyclic subgroups of orders $m$ and $n$ is
    abelian.

\item[\rm(iii)] Every finite group factorisable as an exact product of two cyclic groups of
    orders $m$ and $n$ is isomorphic to $\mathbb Z_m\times
    \mathbb Z_n$.

\item[\rm(iv)]  There is only one $(m,n)$-reciprocal pair of
    skew-morphisms $(\varphi,{\varphi^*})=(\mathrm{id}_n,\mathrm{id}_m)$ of
    the cyclic groups $\mathbb{Z}_n$ and $\mathbb{Z}_m$.

\item[\rm(v)] Up to reciprocality, there is a unique
    isomorphism class of regular dessins whose underlying
    graph is the complete bipartite graph
    $K_{m,n}$.

\item[\rm(vi)] There exists a unique orientable
    edge-transitive embedding of $K_{m,n}$.

\end{itemize}
\end{theorem}
The proof of the equivalence between items (i), (iii) and (vi) of
Theorem~\ref{the:uniq} can be found in \cite[Theorem
1.1]{FL2017}.

\begin{remark}
For a fixed positive integer $x$, it has been recently shown by
Nedela and Pomerance \cite{NP2017} that the number of singular
pairs $(m,n)$ with $m,n\leq x$ is asymptotic to $z(x)^2$ where
\[
z(x)=e^{\gamma}\frac{x}{\log\log\log x},
\]
where $\gamma$ is  Euler's constant.
\end{remark}

\section{The symmetric case}
Recall that a complete regular dessin $\mathcal{D}=(G;a,b)$ is
symmetric if $G$ has an automorphism transposing $a$ and $b$.
In this case  the dessin $\mathcal{D}$ possesses an external
symmetry transposing the colour-classes. If we ignore the
vertex-colouring, the dessin can be regarded as an orientably
regular map with underlying graph $K_{n,n}$.  As a consequence
of Theorem~\ref{thm:corresp-all} we obtain the following
correspondence between regular embeddings of the complete
bipartite graphs $K_{n,n}$ and symmetric skew-morphisms of
$\mathbb{Z}_n$, partially indicated by Kwak and Kwon already in
\cite[Lemma 3.5]{KK2005}.

\begin{corollary}\label{cor:KNNSM}
The isomorphism classes of regular embeddings of complete
bipartite graphs $K_{n,n}$ are in one-to-one correspondence
with the symmetric skew-morphisms of $\mathbb{Z}_n$.
\end{corollary}

A complete classification of orientably regular embeddings of
complete bipartite graphs $K_{n,n}$ has already been
accomplished by Jones et al. in a series of papers
\cite{DJKNS2007, DJKNS2010, DJKNS2013, Jones2010, JNS2007,
JNS2008, NSZ}. The methods used in the classification rely on the
analysis of the structure of the associated exact bicyclic groups. A
different approach to the classification can be taken on the basis
of Corollary~\ref{cor:KNNSM} via determining the corresponding
symmetric skew-morphisms of $\mathbb{Z}_n$. In particular, we
can reformulate Theorem A of \cite{JNS2008} as follows:

\begin{corollary}
The following statements are equivalent for every positive integer $n$:
\begin{itemize}
\item[\rm(i)] The integer $n$ is singular.
\item[\rm(ii)] Every finite group factorisable as a product
    of two cyclic subgroups of order $n$ is abelian.
\item[\rm(iii)] Every finite group  factorisable as an
    exact product of two cyclic subgroups of order $n$ is
    isomorphic to $\mathbb Z_n\times \mathbb Z_n$.
\item[\rm(iv)] The cyclic group $\mathbb{Z}_n$ has only one
    symmetric skew-morphism.
\item[\rm(v)] Up to isomorphism, the complete bipartite
    graph $K_{n,n}$ has a unique orientably regular embedding.
\end{itemize}
\end{corollary}

Although skew-morphisms are implicitly present in the structure
of the automorphism groups of the maps, the way how to find
them explicitly is not at all clear. This leads us to
formulating the following problems for future investigation.

\begin{problem}\label{prob:skewknn}
Determine the symmetric skew-morphisms of cyclic groups by
means of explicit formulae.
\end{problem}

\begin{problem}\label{prob:classregmap}
Classify all orientably regular embeddings of complete
bipartite graphs $K_{n,n}$ in terms of the corresponding
symmetric skew-morhisms.
\end{problem}

The previous problem suggests the following natural question:
under what conditions a symmetric skew-morphism is a group
automorphism and what are the corresponding orientably regular
maps? The following result determines these skew-morphisms
explicitly.

\begin{theorem}
Let $\varphi\colon x\mapsto rx$ be an automorphism of
$\mathbb{Z}_n$ of order $m$, where $\gcd(r,n)=1$. Then
$\varphi$ is a symmetric skew-morphism of $\mathbb{Z}_n$ if and
only if $m\mid n$ and $r\equiv1\pmod{m}$.
\end{theorem}

\begin{proof}
Note that the order of $\varphi$ is equal to the multiplicative
order of $r$ in $\mathbb{Z}_n$. Since
$|\Aut(\mathbb{Z}_n)|=\phi(n)$, we have $m\mid\phi(n)$.  Since
$\varphi$ is an automorphism, the associated power function is
$\pi(x)\equiv1\pmod{m}$ for all $x\in \mathbb{Z}_n$.

If $\varphi$ is symmetric, then by Definition~\ref{def:recipr},
$m\mid n$ and  $\pi(x)=-\varphi^{-x}(-1)\pmod{m}$ for all $x\in
\mathbb{Z}_n$. In particular,
$1=\pi(-1)\equiv-\varphi(-1)\equiv\varphi(1)\equiv r\pmod{m}$.

Conversely, assume that $m\mid n$ and $r\equiv 1 \pmod{m}$. By
Definition~1, it suffices to show that  $-\varphi^{-x}(-1)$ is
a power function of $\varphi$ where $x\in \mathbb{Z}_n$, that
is, to show that $-\varphi^{-x}(-1)\equiv 1\pmod{m}$. Since
$r\equiv 1 \pmod{m}$, we have
$-\varphi^{-x}(-1)=\varphi^{-x}(1)=r^{-x}\equiv 1 \pmod{m}$, as
required.
\end{proof}

Note that if the conditions $m\mid\phi(n)$ of the above theorem
are satisfied, then clearly $m\mid \gcd(n,\phi(n))$ where $\phi$
is the Euler's totient function.

\medskip

The following example shows that there exist symmetric
skew-morphisms of $\mathbb{Z}_n$ which are not automorphisms.

\begin{example}
The cyclic group $\mathbb{Z}_8$ has the total of six
skew-morphisms, four automorphisms and two that are not
automorphisms. The latter two are listed below along with the
corredsponding power functions:
\begin{align*}
\varphi&=(0)(1\,3\,5\,7)(2)(4)(6),&\pi_{\varphi}&=[1]~[3\,3\,3\,3]~[1]\,[1]~[1],
\\
\psi&=(0)(1\,7\,5\,3)(2)(4)(6), & \pi_{\psi}&=[1]~[3\,3\,3\,3]~[1]\,[1]~[1].
\end{align*}
It can be easily verified that all the six skew-morphisms are
symmetric, corresponding to the six non-isomorphic regular
embeddings of $K_{8,8}$ described in \cite[Table 1]{Jones2010}.
\end{example}

%%%%%%%%%%%%%%%%%%%%%%%%%%%  Acknowlegement  %%%%%%%%

\section*{Acknowledgement}
The first author was supported by the Natural Science
Foundation of China (11571035, 11731002). The second author was
supported by Natural Science Foundation of Zhejiang Province
(LY16A010010). The third  author was supported by the grants
APVV-15-0220, VEGA 1/0150/14, Project LO1506 of the Czech
Ministry of Education, Youth and Sports. The fourth author was
supported by the grants APVV-15-0220 and VEGA 1/0813/18. The
fifth author was supported by Natural Science Foundation of
Zhejiang Province (LQ17A010003).

%%%%%%%%%%%%%%%%%%%%%%%%%%%  References  %%%%%%%%%%%%%%%%%%%

%%%%%%%%%%%%%%%%%%%%%%%%%%%%%%%%%%%%%%%%%%%%%%%%%%%%%%%%

\end{document}